\documentclass{article}
\usepackage[left=1in,top=1in,right=1in]{geometry}

\usepackage{amssymb, amsmath, amsthm, graphicx,hyperref}
\usepackage{listings} 

\theoremstyle{definition}
\newtheorem{theorem}{Theorem}[section]
\newtheorem{lemma}[theorem]{Lemma}
\newtheorem{claim}[theorem]{Claim}
\newtheorem{corollary}[theorem]{Corollary}
\newtheorem{proposition}[theorem]{Proposition}
\counterwithin{equation}{section}

\title{Unavoidable patterns and plane paths in dense topological graphs}
\author{Bal\'azs Keszegh\thanks{HUN-REN Alfréd Rényi Institute of Mathematics and ELTE Eötvös Loránd University, Budapest, Hungary.
Supported by the ERC Advanced Grant ``ERMiD'', no.~101054936 and by the EXCELLENCE-24 project no.~151504 Combinatorics and Geometry of the NRDI Fund. Email: {\tt keszegh@renyi.hu}.} \and Andrew Suk\thanks{Department of Mathematics, University of California at San Diego, La Jolla, CA, 92093 USA. Supported by NSF grant DMS-2246847. Email: {\tt asuk@ucsd.edu}.}  \and G\'abor Tardos\thanks{HUN-REN Alfréd Rényi Institute of Mathematics, Budapest, Hungary. Supported by ERC Advanced Grants ``GeoScape'', no. 882971 and ``ERMiD'', no. 101054936. Email: {\tt tardos@renyi.hu}.} \and Ji Zeng\thanks{HUN-REN Alfréd Rényi Institute of Mathematics, Budapest, Hungary. Supported by ERC Advanced Grants ``GeoScape'', no. 882971 and ``ERMiD'', no. 101054936. Email: {\tt zeng.ji@renyi.hu}.}}
\date{}

\begin{document}

\maketitle

\begin{abstract}
Let $C_{s,t}$ be the complete bipartite geometric graph, with $s$ and $t$ vertices on two distinct parallel lines respectively, and all $s t$ straight-line edges drawn between them. In this paper, we show that every complete bipartite simple topological graph, with parts of size $2(k-1)^4 + 1$ and $2^{k^{5k}}$, contains a topological subgraph weakly isomorphic to $C_{k,k}$.  As a corollary, every $n$-vertex simple topological graph not containing a plane path of length $k$ has at most $O_k(n^{2 - 8/k^4})$ edges.  When $k = 3$, we obtain a stronger bound by showing that every $n$-vertex simple topological graph not containing a plane path of length 3 has at most $O(n^{4/3})$ edges. We also prove that $x$-monotone simple topological graphs not containing a plane path of length 3 have at most a linear number of edges.  
\end{abstract}

\section{Introduction}

A \textit{topological graph} is a graph drawn in the plane such that its vertices are represented by points and its edges are represented by non-self-intersecting arcs connecting the corresponding points. No edge is allowed to pass through any point representing a vertex other than its endpoints. Tangencies between the edges are not allowed. That is, if two edges share an interior point, then they must properly cross at this point. A topological graph is called \textit{plane} if there are no crossing edges. Given a topological graph $G$, we say that $H$ is a \textit{topological subgraph} of $G$ if $V(H) \subset V(G)$ and $E(H) \subset E(G)$. We say that $G$ and $H$ are \textit{weakly isomorphic} if there is an incidence preserving bijection between the vertices and edges of $G$ and $H$ such that two edges of $G$ cross if and only if the corresponding edges in $H$ cross as well. A topological graph is \textit{simple} if every pair of its edges intersect at most once: at a common endpoint or at a proper crossing. Simple topological graphs are also known as \textit{simple drawings}. If the edges of a topological graph are drawn with straight-line segments, then it is called \textit{geometric}. We call a geometric graph \textit{convex} if its vertices are in convex position.

In this paper, we are interested in finding large unavoidable patterns in dense simple topological graphs, and in particular, finding large plane paths.  Let us emphasize here that a \emph{path of length $k$} consists of $k + 1$ distinct vertices and $k$ edges.  It is not hard to see that the simple condition here is necessary for plane paths, as one can easily draw $K_n$ in the plane such that every pair of edges cross. Moreover, a construction due to Pach and T\'oth \cite{PT10} shows that there is a drawing of $K_n$ in the plane such that every pair of edges crosses exactly once or twice.  In 2003, Pach, Solymosi, and T\'oth \cite{PST} showed that every complete $n$-vertex simple topological graph contains a topological subgraph on $\Omega(\log^{1/8}(n))$ vertices that is weakly isomorphic to either a complete convex geometric graph or a so-called complete twisted graph. This bound was later improved by Suk and Zeng \cite{SZ} to $(\log n)^{1/4 - o(1)}$.  In 1998, Negami proved a bipartite analogue of this theorem.  Let $C_{s,t}$ be a complete bipartite geometric graph with vertex sets $U$ and $V$, where $|U| = s$ and $|V| = t$, such that the vertices in $U$ lie on the $y$-axis and the vertices in $V$ lie on the vertical line $x = 1$. (It is easy to see that $C_{s,t}$ is determined up to weak isomorphism independent of the exact placement of the vertices on these vertical lines.) In \cite{Ne}, Negami showed that for every $k > 1$, there is a minimum integer $n = n(k)$ such that every complete bipartite simple topological graph with $n$ vertices in each part contains a topological subgraph weakly isomorphic to $C_{k,k}$. The proof in \cite{Ne} is based on 4-uniform hypergraph Ramsey theory and no explicit bound for $n(k)$ is given. By applying more geometric arguments, we establish the following stronger result.
 
\begin{theorem}\label{main}
Every complete bipartite simple topological graph with vertex sets $U$ and $V$, where $|U| > 2(k-1)^4$ and $|V| \geq 2^{k^{5k}}$, contains a topological subgraph weakly isomorphic to $C_{k,k}$.
\end{theorem}

\noindent We suspect that Theorem \ref{main} still holds if $|V|$ is at least single exponential in a power of $k$, rather than double exponential in $\Omega(k\log k)$.  On the other hand, our next result shows that the size of $U\cup V$ needs to be at least exponential in $k/2$.

\begin{proposition}\label{main2}
For every positive integer $k$ there exists a complete simple topological graph on $2^{\lfloor k/2\rfloor}$ vertices that does not contain a topological subgraph weakly isomorphic to $C_{k,k}$.
\end{proposition}

\noindent Very recently, Balogh, Parada, and Salazar~\cite{balogh2025unavoidable} extended Negami's result to complete multipartite simple topological graphs. Let us remark that our quantitative bound for the bipartite case, i.e.~Theorem~\ref{main}, is particularly interesting, because we can combine it with the celebrated K\H{o}v\'ari-S\'os-Tur\'an theorem to conclude the following statement.

\begin{theorem}\label{main3}
    Every $n$-vertex simple topological graph without a topological subgraph weakly isomorphic to $C_{k,k}$ has at most $O_k(n^{2-\epsilon(k)})$ edges with $\epsilon(k) = (2(k-1)^4+1)^{-1}$.
\end{theorem}

\noindent It is an interesting open question to determine the optimal value of $\epsilon(k)$ in the theorem above.

\smallskip 

An old conjecture due to Rafla \cite{Raf} predicts that every complete $n$-vertex simple topological graph contains a plane Hamiltonian cycle for $n \geq 3$. This has been verified for small $n$ and is open for $n \geq 10$ (see~\cite{Ab}). The much weaker conjecture of finding a linear size plane matching in a complete $n$-vertex simple topological graph has been extensively studied \cite{S2,Ai,RV}. The best known result for this problem is due to Aichholzer--García--Tejel--Vogtenhuber--Weinberger~\cite{Ai}, who showed the existence of a plane matching of size $\Omega(\sqrt{n})$. Even less is known about plane paths. Recently, it is proved~\cite{Ai,SZ} that every complete $n$-vertex simple topological graph contains a plane path of length $\Omega(\log n/\log\log n)$. Since one can easily find a plane path of length $k$ in $C_{\lceil k/2\rceil,\lceil k/2\rceil}$, we can apply Theorem \ref{main3} to obtain the following corollary.

\begin{corollary}\label{no-plane-k-path}
    Every $n$-vertex simple topological graph not containing a plane path of length $k$ has at most $O_k(n^{2 - 8/k^4})$ edges.
\end{corollary}

\smallskip 

Next, we concentrate on the case of Corollary~\ref{no-plane-k-path} when $k=3$, that is, on simple topological graphs with no plane path of length $3$. Such graph drawings can be regarded as a local variant of the so-called thrackles. A \textit{thrackle} is a simple topological graph such that every pair of edges that do not share a vertex must cross. The famous thrackle conjecture of Conway states that a thrackle with $n$ vertices has at most $n$ edges, see \cite{PachSterling} for a detailed history of the problem. For the special case of geometric graphs, the thrackle conjecture is proved by Erd\H{o}s, and later, a short proof is given by Perles (see \cite{PachSterling}). The proof of Perles also works if the drawing is outerplanar \cite{cairns2012outerplanar}, i.e., if the points lie on the boundary of a disk and the edges lie inside this disk. The case that the edges are $x$-monotone is settled by Pach and Sterling~\cite{PachSterling}. In general, after a long series of improvements  (Lov\'asz--Pach--Szegedy \cite{lovasz1997conway} proved $2n$, Cairns--Nikolayevsky \cite{cairns2000bounds} proved $1.5 n$, improved further by Fulek--Pach \cite{fulek2011computational,fulek2019thrackles} and Goddyn--Xu \cite{goddyn2017bounds}), the current best upper bound on the number of edges of an $n$-vertex thrackle is $1.393(n-1)$ by Xu \cite{xu}. Apart from improvements on the upper bound, several papers consider variations of thrackles, in particular generalized thrackles \cite{cairns2000bounds}, superthrackles \cite{superthrackles}, thrackles on surfaces \cite{thracklesnonplanarsurfaces}, and thrackles of convex sets \cite{asada2018reay, keszegh2023convexhullthrackles}. 

We call a simple topological graph with no plane path of length $3$ a \textit{local thrackle}. Note that every thrackle is a local thrackle, but the two notions do differ.  For example, consider the graph $G$ on 6 vertices and 7 edges obtained from two disjoint copies of $K_3$ and an edge between them.  Then $G$ does not have a thrackle drawing (see, e.g., \cite{woodall1971thrackles}), but $G$ does have a local thrackle drawing (see the left-most graph in Figure~\ref{fig:2triangles}).  However, our next result shows that the number of edges in a local thrackle with straight-line edges cannot be larger than $n$, just like in the case of thrackles. 

\begin{proposition}\label{lt:geometric}
   A local thrackle on $n$ vertices and with straight-line segments as edges has at most $n$ edges.
\end{proposition}

 \noindent Our proof in fact gives a characterization of the graphs that have a local thrackle drawing with segments. On the other hand, already with $x$-monotone edges one can draw a local thrackle with more than $n$ edges, unlike in the case of thrackles.

\begin{proposition}\label{lt:xmon}
    A local thrackle on $n$ vertices and with $x$-monotone edges has at most $\frac{4}{3}n$ edges, while there exist local thrackles on $n$ vertices with $x$-monotone edges and with $\frac{6}{5}n-O(1)$ edges. 
\end{proposition}

\noindent For the general case, we have a slightly better construction and an upper bound which improves the bound that comes from Corollary~\ref{no-plane-k-path}.

\begin{proposition}\label{lt:general}
    A local thrackle on $n$ vertices has at most $O(n^{4/3})$ edges, while there exist local thrackles on $n$ vertices and with $2n-O(\sqrt n)$ edges. 
\end{proposition}

\noindent Notice that the lower bound construction for the general case has more edges than the upper bound for the $x$-monotone case.

\smallskip
\noindent\textbf{Avoiding self-intersecting paths.}
A \textit{dual} version of Proposition~\ref{lt:xmon} was obtained by Pach, Pinchasi, Tardos, and T\'oth in \cite{PPTT}, who proved that every $n$-vertex $x$-monotone simple topological graph with no self-intersecting path of length 3 has at most $O(n\log n)$ edges.  Moreover, they showed that this bound is asymptotically tight.  For fixed $k \geq 3$, a construction due to Tardos \cite{T} shows that there are geometric graphs with no self-intersecting path of length $k$ and with a superlinear number of edges. Pach, Pinchasi, Tardos, and T\'oth in \cite{PPTT} also showed that every $n$-vertex topological graph (not necessarily simple) with no self-intersecting path of length 3 has at most $O(n^{3/2})$ edges.  In Section \ref{sec:4}, we give a short proof of the following variant.

\begin{theorem}\label{thm:selfcross}
Every $n$-vertex simple topological graph $G$ with no self-intersecting path of length 4 has at most $O(n\log^2 n)$ edges.
\end{theorem}

Our paper is organized as follows. In the next section, we prove Theorem~\ref{main} and Proposition~\ref{main2}. In Section~\ref{sec:local}, we establish our results on local thrackles, proving Propositions~\ref{lt:geometric}, \ref{lt:xmon}, and \ref{lt:general}. In Section \ref{sec:4} we prove Theorem \ref{thm:selfcross}. Finally, in Section \ref{sec:remarks}, we give some concluding remarks.  We systematically omit floors and ceilings whenever they are not crucial for the sake of clarity in our presentation. All logarithms are in base 2.

\section{Unavoidable patterns in complete bipartite topological graphs}

In this section, we prove Theorem \ref{main}.  We will need the following combinatorial lemmas.  Let $G$ be an ordered graph with vertex set $U = \{u_1,u_2,\ldots,u_n\}$.  We say that $G$ contains a \textit{monotone path} of length $k$ if there are $k +1$ indices $1\le i_0 < i_1<\cdots < i_k\le n$ such that $u_{i_j}u_{i_{j+1}} \in E(G)$ for all $0\leq i< k$.  We say that $G$ is \textit{transitive} if for any $1\le i_1 < i_2 < i_3\le n$, the condition $u_{i_1}u_{i_2},u_{i_2}u_{i_3} \in E(G)$ implies that $u_{i_1}u_{i_3} \in E(G)$. We will need the following three simple observations. The first is self evident, for the second we provide a simple proof to be self contained. The last one is a well known estimate of the multicolor Ramsey numbers.

\begin{lemma}\label{lemtrans}
     The vertices of a monotone path induce a clique in a transitive ordered graph.
\end{lemma}

\begin{lemma}\label{lemES}
    Let $t$ and $m$ be positive integers. If $G$ is a complete ordered graph on $n>tm^4$ vertices whose edges are colored with five colors, then either there exists an $m$-edge monochromatic monotone path in $G$ in one of the first four colors, or there exists a monochromatic clique of size $t+1$ in the fifth color.
\end{lemma}

\begin{proof}
Label each vertex $v$ of $G$ by the four-tuple $(l_1,l_2,l_3,l_4)$, where $l_i$ is the edge number of the longest monotone path in color $i$ ending at $v$. If any of these paths have at least $m$ edges we are done. If this is not the case, then all the labels are from $\{0,\dots,m-1\}^4$, so by the pigeonhole principle, some $t+1$ vertices must share the same label. We claim that these vertices induce a monochromatic clique in the fifth color. Indeed, assume the vertex $u$ precedes vertex $v$ in the vertex order and the edge $uv$ receives color $i$. Consider the longest monotone path in color $i$ ending at $u$. Extending this path with the $uv$ edge we get a longer monotone path in color $i$ ending at $v$ showing that $u$ and $v$ cannot share the same label unless $i=5$. This proves the lemma.
\end{proof}

Let $r(k;r)$ be the minimum integer $n$ such that every $r$-coloring of the edges of the complete graph on $n$ vertices contains a monochromatic clique on $k$ vertices. The existence of $r(k;r)$ follows by the famous theorem of Ramsey \cite{R30}, and following the arguments of Erd\H{o}s--Szekeres \cite{ES}, we have
\begin{lemma}\label{multicolor_ramsey}
$r(k;r) < r^{rk}.$
\end{lemma}

\begin{proof}[Proof of Theorem~\ref{main}]
Let $G$ be a complete bipartite simple topological graph whose vertex set is $U \cup V$, where $|U|  = m > 2(k-1)^4$ and  $|V| = n \geq 2^{k^{5k}}$. We assume without loss of generality that $m = 2(k-1)^4+1$ (otherwise we simply ignore the extra vertices in $U$). There is nothing to prove when $k=1$. When $k=2$, since $K_{3,3}$ is not planar, we can find a pair of crossing edges $u_1v_1$ and $u_2v_2$ in $G$, and as $G$ is simple, the induced topological subgraph on $u_1,u_2,v_1,v_2$ is weakly isomorphic to $C_{2,2}$. Therefore, we assume $k\geq 3$ for the rest of this proof.
 
By arbitrarily ordering the elements in $U$ and $V$, we have $U = \{u_1,\ldots, u_m\}$, and $V = \{v_1,\ldots, v_n\}$. For each 4-tuple $u_i,u_j \in U$ and $v_s, v_t \in V$, where $i < j$ and $s < t$, let us consider the drawing of $K_{2,2}$ induced on these vertices. Note that as we have a simple topological graph, only non-adjacent edges can cross, that is, $u_iv_s$ may cross $u_jv_t$ or $u_iv_t$ may cross $u_jv_s$. If both of these pairs of edges cross, then we obtain a thrackle drawing of $C_4$ and a small case-analysis shows that such a drawing does not exist. So we have three possibilities, either the first pair cross, or the second pair, or it is a plane drawing. We further refine this classification according to the orientation of the vertices around the crossing point obtaining the following five \textit{order types}. When speaking about the order of the vertices around the crossing point we mean the cyclic order of the edge-segments going to these vertices. 

\begin{enumerate}

    \item The 4-tuple $\{u_i,u_j,v_s,v_t\}$ is of \textit{type 1} if edge $u_iv_t$ crosses edge $u_jv_s$, and moreover, vertices $(u_i,v_s,v_t,u_j)$ appear in clockwise order from the crossing point of $u_iv_t$ and $u_jv_s$.  

    \item  The 4-tuple $\{u_i,u_j,v_s,v_t\}$ is of \textit{type 2} if edge $u_iv_t$ crosses edge $u_jv_s$, and moreover, vertices $(u_i,u_j,v_t,v_s)$ appear in clockwise order from the crossing point of $u_iv_t$ and $u_jv_s$.  

    \item  The 4-tuple $\{u_i,u_j,v_s,v_t\}$ is of \textit{type 3} if edge $u_iv_s$ crosses edge $u_jv_t$, and moreover, vertices $(u_i,v_t,v_s,u_j)$ appear in clockwise order from the crossing point of $u_iv_s$ and $u_jv_t$. 

    \item  The 4-tuple $\{u_i,u_j,v_s,v_t\}$ is of \textit{type 4} if edge $u_iv_s$ crosses edge $u_jv_t$, and moreover, vertices $(u_i,u_j,v_s,v_t)$ appear in clockwise order from the crossing point of $u_iv_s$ and $u_jv_t$. 

    \item  The 4-tuple $\{u_i,u_j,v_s,v_t\}$ is of \textit{type 5} if it induces a plane drawing of $K_{2,2}$.

\end{enumerate}

\begin{figure}[ht]
    \centering
    \includegraphics[width=\linewidth]{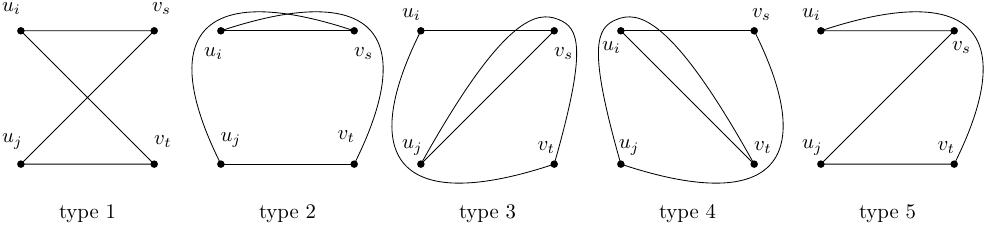}
    \caption{A drawing of each of the five types of $C_4$'s.}
    \label{fig:types}
\end{figure}

\noindent See Figure \ref{fig:types}. For $v_s,v_t \in V$ and $w \in \{1,2,3,4,5\}$, let $G^{(w)}_{s,t}$ be the ordered graph whose vertex set is $U$, and edges are pairs of vertices $\{u_i,u_j\}$ such that the 4-tuple $\{u_i,u_j,v_s,v_t\}$ is of type $w$.

We claim that the ordered graph $G^{(w)}_{s,t}$ is transitive for $w \in \{1,2,3,4\}$, that is, with respect to the first four order types. We verify the statement for $G^{(1)}_{s,t}$. For $i < j < \ell$, suppose that $\{u_i,u_j\}$ and $\{u_j,u_\ell\}$ are edges in $G^{(1)}_{s,t}$. Consider the drawing of $K_{2,2}$ induced by the 4-tuple $\{u_i,u_j,v_s,v_t\}$, specifically the ``triangle'' formed by $u_j$, $v_t$ and the crossing point between $u_jv_s$ and $u_iv_t$. Since $(u_j,v_s,v_t,u_\ell)$ must appear in clockwise order from the crossing point of $u_jv_t$ and $u_\ell v_s$, we must exit the ``triangle'' at this crossing point as we travel along $u_{\ell}v_s$ from $v_s$ to $u_\ell$. This implies that $u_{\ell}v_s$ crosses both $u_jv_t$ and $u_iv_t$. Further, $u_\ell$, $v_s$, and $v_t$ must appear in the same cyclic order from the crossing point of $u_jv_t$ and $u_\ell v_s$ as well as from the crossing point of $u_iv_t$ and $u_\ell v_s$. The former is clockwise order from the assumption $\{u_j,u_\ell\}\in E(G^{(1)}_{s,t})$, so the latter also clockwise implying $\{u_i,u_\ell\}\in E(G^{(1)}_{s,t})$ as claimed. See Figure \ref{fig:type1}. The transitivity property of $G^{(2)}_{s,t}, G^{(3)}_{s,t}$, and $G^{(4)}_{s,t}$ can be argued similarly, or be reduced to the case of $G^{(1)}_{s,t}$ by reversing the vertex orders on $U$ or $V$. For example, a type 3 quadruple $\{u_i,u_j,v_s,v_t\}$ would have been type 1 if we had ordered the vertices in $U$ the same but in $V$ reversely.

\begin{figure}[ht]
    \centering
    \includegraphics[width=0.8\linewidth]{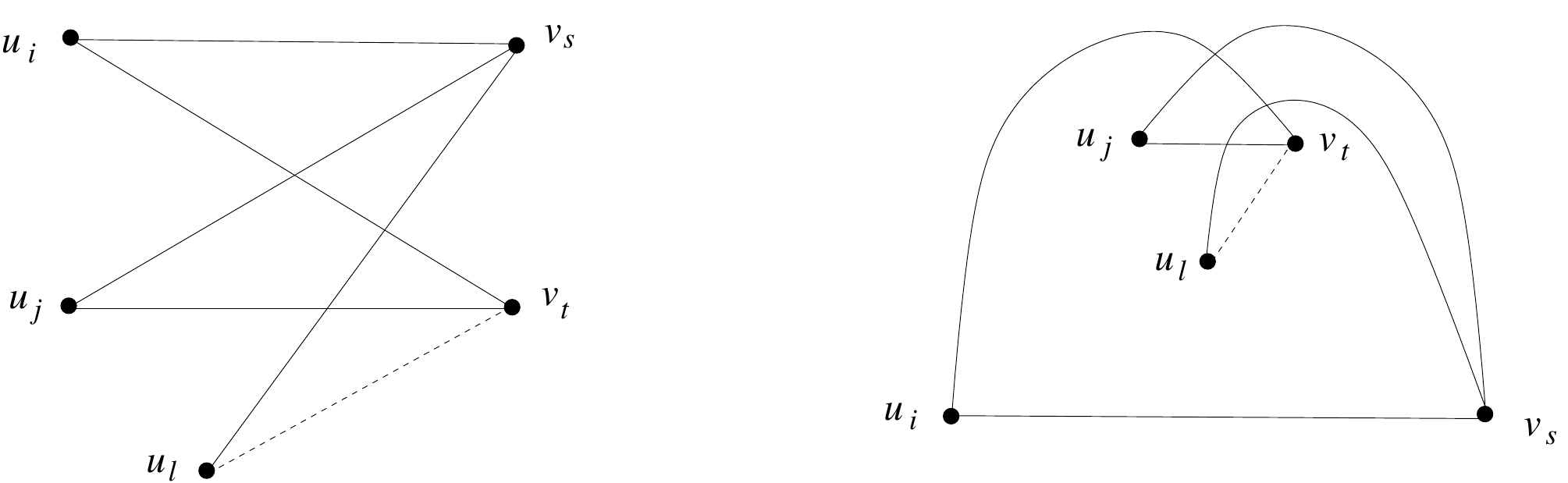}
    \caption{Proof of transitivity of $G^{(1)}_{s,t}$.}
    \label{fig:type1}
\end{figure}

Since $|U| > 2(k-1)^4$, by combining Lemmas \ref{lemtrans}, \ref{lemES}, and the fact that $G_{s,t}^{(w)}$ is transitive when $w \in \{1,2,3,4\}$, we have the following observation.
\begin{claim}
    For each pair $v_s,v_t \in V$, there is a subset $S\subset U$ such that either
    \begin{enumerate}
    
        \item $|S| = k$, and $S$ induces a clique in $G_{s,t}^{(w)}$ for some $w \in \{1,2,3,4\}$, or

        \item $|S| = 3$, and $S$ induces a clique (i.e., a triangle) in $G_{s,t}^{(5)}$. 

    \end{enumerate}
\end{claim}

This claim implies that we can label the edge $v_sv_t$ of the complete graph on $V$ with a pair $(S,i)$, where $i \in \{1,\ldots, 5\}$ and $S \subset U$ is a clique in $G^{(i)}_{s,t}$ of size $3$ if $i=5$ or of size $k$ otherwise. In case there are more possible labels we choose arbitrarily. For $k\geq 7$, using $|U| = 2(k-1)^4+1 \leq 2k^4$, the number of distinct labels we use on the edges is at most \begin{equation*}
    r(k) := 4\binom{|U|}{k} + \binom{|U|}{3} \leq 5\binom{2k^4}{k} \leq 5(2ek^3)^k \leq k^{4k}.
\end{equation*} And we can further estimate using $k \geq 7$ \begin{equation*}
    r(k)^{r(k) \cdot k} < 2^{k^{5k}}.
\end{equation*} For $k = 3,4,5,6$, the upper bound for $r(k)$ is too crude, but we can verify the previous inequality by direct computation.

Applying Lemma~\ref{multicolor_ramsey} together with $|V| \geq 2^{k^{5k }}$, there is a subset $T \subset V$ of size $k$ such that every pair in $T$ has the same label $(S,i)$. Let us consider the topological subgraph $G'$ induced on $S \cup T$. The proof now falls into the following cases.

\smallskip

\noindent Case 1.  Suppose $|S| = k$ and $i \in\{1,2,3,4\}$. We can assume that $i = 1$, since a symmetric argument would follow otherwise.  Then for $u_i,u_j \in S$ and $v_s,v_t \in T$, where $i < j$ and $s < t$, $\{u_i,u_j,v_s,v_t\}$ is of type~1, and therefore, $u_iv_t$ crosses $u_jv_s$.  Hence, $G'$ is weakly isomorphic to $C_{k,k}$.  

\smallskip

\noindent Case 2.  Suppose $|S| = 3$ and $i = 5$.  Then $G'$ is a plane drawing of $K_{3,k}$. This is impossible as $k \ge 3$.    
\end{proof}

\subsection{Construction: Proof of Proposition~\ref{main2}}
Here we show that one needs exponentially many vertices in a simple topological graph---even in a simple complete topological graph---to ensure the presence of a subgraph weakly isomorphic to $C_{k,k}$. Note however, that this lower bound does not match the doubly exponential bound in Theorem~\ref{main}.

Fix $k\ge3$. We construct a large random simple complete topological graph as follows, and show that with positive probability it has no subgraph weakly isomorphic to $C_{k,k}$. Set $n  = 2^{\lfloor k/2\rfloor}$. Let $U$ be a set of $n$ points placed on the $x$-axis. This will be the vertex set of the graph we construct. For each pair $u,v\in U$, 
draw an edge connecting $u$ and $v$ as a half-circle in the upper or
lower half of the plane uniformly at random and independently for each pair $(u,v)$. Clearly, this will result in a simple drawing of $K_n$ in the plane.  

Consider four vertices $u_1,u_2,u_3,u_4\in U$ that appear in this order on the $x$ axis. These vertices induce a drawing of $K_4$ in which only the edges $u_1u_3$ and $u_2u_4$ cross if and only if they are drawn in the same half plane. If these two edges are drawn in distinct half planes we obtain a plane drawing of $K_4$.

Let $S,T\subseteq U$ be disjoint sets of size $k$ and let us calculate the probability that the edges between $S$ and $T$ form a subgraph weakly isomorphic to $C_{k,k}$. First note that $C_{k,k}$ has no subgraph that is a plane drawing of $C_4$, so for the weak isomorphism we need that for every $u\ne u'\in S$ and $v\ne v'\in T$ some two edges between $\{u,u'\}$ and $\{v,v'\}$ cross. This implies that the order on $x$ axis cannot alternate between the vertices in $S$ and the vertices of $T$, because in that case the subgraph is a plane drawing with probability one. This implies that either $S$ can be split into two sets $S_1$ and $S_2$, so that all vertices in $S_1$ are to the left of the vertices of $T$ and all vertices in $S_2$ are to the right of the vertices in $T$, or $T$ can be split in two sets, one to the left and one to the right of $S$. We call such pairs \textit{separated}.

Now let $S$ and $T$ be two separated $k$-element set of vertices. We may assume that $S$ splits into $S_1$ and $S_2$ as above and $S_1$ is not empty (otherwise, switch the roles of $S$ and $T$).
Let $v_0$ and $v_1$ be the leftmost and rightmost vertices in $T$, respectively, while $u_1$ is the rightmost vertex in $S_1$ and $u_0$ is the leftmost vertex in $S_2$, or in case $S_2$ is empty, then $u_0$ is the leftmost vertex in $S_1$. We claim that the subgraph formed by the edges between $S$ and $T$ is weakly isomorphic to $C_{k,k}$ if and only if all of its edges, with the possible exception of $u_1v_0$ and $u_0v_1$, are drawn in the same (upper or lower) half plane. The if part of this equivalence is easy to check (and strictly speaking, not even needed for our argument). For the only if part assume without loss of generality that $u_1v_1$ is drawn in the upper half plane. Now consider two other vertices: $u_1\ne u\in S$ and $v_1\ne v\in T$. If $uv$ is drawn in the lower half plane, then the $4$-cycle $u_1v_1uv$ is a plane drawing, so we do not have the desired weak isomorphism. So the weak isomorphism implies that all of these edges are drawn in the upper half plane including the edge $u_0v_0$. But then a similar argument shows that all the edges $uv$ with $u_0\ne u\in S$ and $v_0\ne v\in T$ must also be drawn in the upper half plane. This covers all the edges mentioned in the claim. 

As a result, the probability that the edges between the separated sets $S$ and $T$ form a subgraph weakly isomorphic to $C_{k,k}$ is exactly $2^{-k^2+3}$ and therefore the expected number of subgraphs weakly isomorphic to $C_{k,k}$ is exactly
$$k\binom{n}{2k}2^{-k^2+3}.$$

By our choice of $n =2^{\lfloor k/2\rfloor}$ and $k\geq 3$, this is less than $1$, so the random drawing of $K_n$ contains no subgraph weakly isomorphic to $C_{k,k}$ with positive probability.

\section{Local thrackles}\label{sec:local}

We say that a graph $G$ is \emph{thrackleable} (\emph{local thrackleable}) if $G$ has a thrackle drawing (local thrackle drawing).  Likewise, we say that $G$ is \emph{geometric local thrackleable} if there is a drawing of $G$ with straight-line edges that is a local thrackle.  

Let us first consider the case when edges are drawn as straight-line segments. In this case, we can even classify all geometric local thrackleable graphs following the arguments of Woodall~\cite{woodall1971thrackles}. We prove Proposition \ref{prop:geometric} below which implies Proposition~\ref{lt:geometric}. A \textit{caterpillar} is a tree in which each vertex is adjacent to at most two non-leaf vertices. A \textit{spiked cycle} is a connected graph with one cycle such that every vertex not on this cycle has degree one. The \textit{length} of a spiked cycle is the length of its unique subgraph that is a cycle. Let us first  recall Woodall's characterization of geometric thrackles \cite{woodall1971thrackles} using these definitions:

\begin{theorem}\label{thm:woodall}
A graph can be drawn as a thrackle with straight-line segments as edges if and only if either it is a union of disjoint caterpillars, or it is a spiked cycle of odd length with some additional isolated vertices.
\end{theorem}

Our characterization of geometric local thrackles is the following:

\begin{proposition}\label{prop:geometric}
    A graph can be drawn as a local thrackle with straight-line segments as edges if and only if it is the disjoint union of caterpillars, spiked cycles of odd length and spiked cycles of even length at least eight.
\end{proposition}
\begin{proof}

    First, we argue that if a graph $G$ contains a vertex $v$ with at least three non-leaf neighbors, then $G$ is not geometric local thrackleable. See Figure~\ref{fig:geometric_fobid} for two such graphs with $G_1$ already covered in \cite{woodall1971thrackles}. Indeed, if $v$ has three non-leaf neighbors, then in any straight line drawing, one of them ``splits'' the neighborhood of $v$, that is, $v$ has neighbors in both sides of the line of the edge $vw$ for some non-leaf neighbor $w$. But as $w$ is no leaf it has another neighbor $w'\ne v$ and in any local thrackle drawing $ww'$ intersects (crosses or has a common end point with) every edge incident to $v$, which is only possible if $w$ does not split the neighborhood, a contradiction.

    \begin{figure}[ht]
        \centering
        \includegraphics[width=0.5\linewidth]{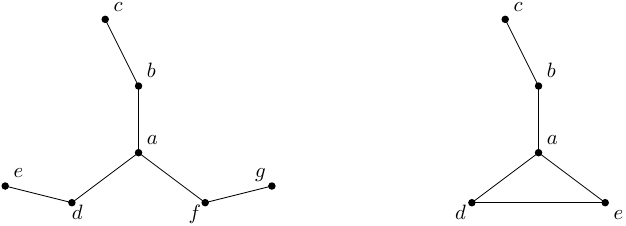}
        \caption{Two forbidden subgraphs $G_1$ (left) and $G_2$ (right) for geometric local thrackles.}
        \label{fig:geometric_fobid}
    \end{figure}
    
    Let us remove all the leaves of a geometric local thrackleable graph $G$. The observation above implies that the remaining graph has maximum degree at most $2$, so it is the disjoint union of paths and cycles. This implies that $G$ is the disjoint union of caterpillars and spiked cycles. To finish the only if part of the proposition we are left to show that $C_4$ and $C_6$ are not geometric local thrackleable. It is easy to see that if $C_4$ is drawn as a local thrackle, then this drawing must be a thrackle. The same holds true for $C_6$ as we shall prove in Lemma~\ref{lem:c6} (although for straight line drawings the case analysis could be shortened). On the other hand, it follows from Theorem \ref{thm:woodall} that even cycles are not geometric thrackleable ($C_4$ is not even thrackleable), hence $C_4$ and $C_6$ are not geometric local thrackleable.

    For the if part of proposition, notice that the components of $G$ can be drawn pairwise disjointly, so it is enough to prove that caterpillars, odd spiked cycles and even spiked cycles of length at least $8$ are geometric local thrackleable. The first two cases are already handled by Theorem \ref{thm:woodall} (as those graphs are geometric thrackleable), so here we concentrate on even spiked cycles. Let $n \geq 8$ be even. We first place vertices $v_1,v_2,\dots,v_{n-3}$ on the circle according to the following clockwise order: $v_1,v_3,v_5,\dots, v_{n-3}$, $v_2, v_4, \dots, v_{n-4}$. Then we put $v_{n-2}$ on the clockwise arc $v_1v_3$, $v_{n-1}$ on the clockwise arc $v_2v_4$, and $v_{n}$ on the clockwise arc $v_{n-5}v_{n-3}$. Observe that we obtain a geometric local threackle drawing of $C_n$ by connecting $v_iv_{i+1}$ using straight-line segments. See Figure~\ref{fig:c8} for an illustration. Finally, we can choose an interval on the cycle for each $v_i$ such that if we place the leaf neighbors of $v_i$ on this intervals the resulting drawing is still a local thrackle.
\end{proof}


\begin{figure}[ht]
        \centering
        \includegraphics{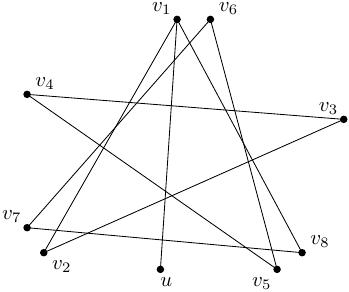}
        \caption{$C_8$ can be drawn as a linear local thrackle, and with a spike $u$ attached to $v_1$.}
        \label{fig:c8}
\end{figure}

We continue with the case when edges are drawn $x$-monotone.

\begin{proof}[Proof of Proposition~\ref{lt:xmon}]
For the lower bound we give a construction with $6/5n-O(1)$ edges. See Figure~\ref{fig:xmon}. First, we draw a path as seen on the top of Figure~\ref{fig:xmon} as a local thrackle. Then we make a translated copy of this path as seen on the middle of Figure \ref{fig:xmon}, this is clearly still a local thrackle. Finally, we add further paths of length $2$, one of them is drawn blue on the middle of Figure \ref{fig:xmon}, and moreover, one can check that adding this one path maintained that the drawing is a local thrackle. All such blue paths are drawn on the bottom of Figure \ref{fig:xmon}, where it is not hard to check that the drawing gives a simple topological graph and also it is still a local thrackle (the blue $2$-paths are far enough in the graph, thus no planar $3$-path can be created when adding all blue $2$-paths, if adding one did not create a planar $3$-path), finishing the construction. Note that we connected using a $2$-path every second vertex of one path with the corresponding vertex of the second path, thus when extending the construction with $2$ vertices on each path plus a vertex to add the $2$-path joining, we added $5$ vertices and $6$ edges. Thus, the number of edges is $6/5n-O(1)$ (the constant error comes from the fact that close to the two ends of the two paths some edges are missing).

\smallskip

For the upper bound, let $G$ be a local thrackle with $x$-monotone edges. We may assume that every vertex has degree at least one. We distinguish three types of vertices. A vertex that has no edges going to the left (resp. right) is called a \textit{left-vertex} (resp. \textit{right-vertex}). A vertex that has edges in both directions is called a \textit{middle-vertex}. Assume there are $l$ left-vertices, $r$ right-vertices and $m$ middle-vertices, $l+r+m=n$.

At each left-vertex, delete the incident edge that leaves highest at this vertex. Delete at each right-vertex the incident edge that leaves lowest at this vertex. We have deleted at most $l+r$ edges.
Notice that if there is a remaining edge between a left-vertex and a right-vertex, then together with the $2$ deleted edges incident to these vertices, we get a plane $3$-path, a contradiction. Therefore, no edge is left between a left- and a right-vertex. Similarly, if a middle-vertex is connected to two left-vertices in the remaining graph, then adding the deleted edge of the left-vertex on the higher of these two edges, we obtain a plane $3$-path. The contradiction shows that every middle vertex is connected to at most one left-vertex in the remaining graph.

Next, we claim that the remaining edges connecting the left-vertices and middle-vertices form a matching. After the observation above it is enough to show that any left-vertex has at most one middle-vertex neighbor. Indeed, otherwise take two edges from a left-vertex to middle-vertices and two more edges going right from these middle-vertices. It is easy to see that one of the $3$-paths in the obtained $4$-path or $4$-cycle must be plane, a contradiction. Therefore, the number of remaining edges connecting left- and middle vertices is at most $\min(l,m)$.

Similarly, there are at most $\min(r,m)$ edges between the right-vertices and middle-vertices left in the graph. No edge connects two middle-vertices because that would lead to an $x$-monotone (therefore plane) $3$-path. Putting everything together, our original graph has at most $l+r+\min(l,m)+\min(r,m)$ edges. Assume without loss of generality that $r\le l$. Then this is at most $l+r+m+\min(l,r,m)=n+\min(l,r,m)\le n+n/3=4n/3$.
\end{proof}

\begin{figure}
	\centering
	\includegraphics[width= 16cm]{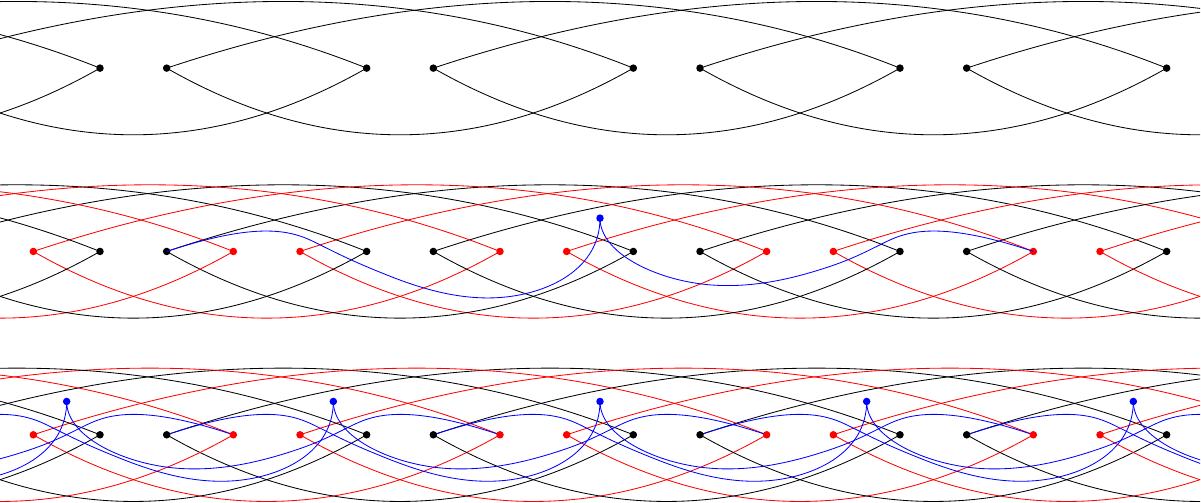}
	\caption{Two paths (black and red) joined by additional $2$-paths (blue) drawn as a local thrackle with $x$-monotone edges.}
	\label{fig:xmon}
\end{figure}

For the case of general local thrackles we need the following statements. Let $\Theta_3$ denote a graph consisting of two vertices connected by three internally-disjoint
paths of length three.  That is,  $\Theta_3$ is a graph with 8 vertices and 9 edges.

\begin{theorem}[\cite{faudree1983class}]\label{thm:notheta3comb}
    A graph that does not contain $\Theta_3$ as a subgraph has at most $O(n^{4/3})$ edges.
\end{theorem}

\begin{lemma}\label{lem:c6}
    Any local thrackle drawing of $C_6$ is also a thrackle drawing of $C_6$.
\end{lemma}

\begin{proof}
Fulek and Pach \cite{fulek2011computational} invented a computational approach to check whether a given graph is thrackleable. In fact, as we shall describe, their method is more general and can check whether any crossing pattern can be realized as a simple drawing or not. Let $G=(V,E)$ be an oriented graph. (We need the orientation for reference purposes only.) Suppose we are given an ordered list $\ell(e)$ consisting of elements in $E\setminus e$ for each $e\in E$, with the property that $f \in \ell(e)$ if and only if $e\in \ell(f)$, then we can construct an auxiliary graph $H$ through the following procedure: \begin{enumerate}
    \item (Planarization) For every $e\in E$ directed from $u$ to $v$, we replace the edge $e = \{u,v\}$ by a path from $u$ to $v$, through the vertices named as $\{e,f\}$ with $f\in \ell(e)$ following the order of $\ell(e)$. Here, we introduce new vertices $\{e,f\}$ if necessary.

    \item (Subdivision) For each edge between two new vertices created in the planarization step, we replace it with a path of length two, by inserting an additional vertex.

    \item (Bracing) For each new vertex $v = \{e,f\}$ created in the planarization step, there will be exactly four vertices $a,b,c,d$ adjacent to $v$ (they may be original vertices in $V$ or new vertices introduced in subdivision step), where $a,b$ are adjacent to $v$ along segments of of $e$ and $c,d$ are adjacent to $v$ along segments of $f$. We create the edges $\{a,c\}$, $\{a,d\}$, $\{b,c\}$, and $\{b,d\}$.
\end{enumerate} We remark that here $\{e,f\}$ is unordered, so it will be considered in the planarization steps for $e$ and $f$ respectively, which will give it exactly four neighbors.

\begin{claim}
    The following two statements are equivalent: (i) There exists a drawing of $G$ where the crossings on each edge $e$, along its direction, are exactly those between $e$ and $f$ with $f\in \ell(e)$ in this order. (ii) The graph $H$ constructed above according to $G$ and $\ell$ is planar.
\end{claim}

We give a brief explanation to this claim. Suppose $G$ has such a drawing, then we can turn it into a plane drawing by turning each crossing point into a new vertex (planarization). Next, we can subdivide each edge between the crossing points once (subdivision). Finally, we connect the neighbors of each crossing point in their cyclic order (bracing). This can be done easily such that the resulting drawing is still plane, implying that $H$ is planar; On the other hand, suppose we have a plane drawing of $H$, we can draw any edge $e$ of $G$, directed from $u$ to $v$, along the path from $u$ to $v$ in $H$ created in planarization. The vertices created in the subdivision step serve as a guide on which path to follow at a crossing point. The bracing structures will enforce the edges drawn to cross properly without tangencies. This will give a drawing of $G$ as wanted.

Following this approach, we used a computer program to verify Lemma~\ref{lem:c6}. To prove this lemma, it suffices for us to take $G = C_6$, direct its edges arbitrarily, create every possible crossing list $\ell$ that obeys the local thrackle condition, construct $H$ and test for its planarity. If all such $H$ except for those corresponding to thrackles are not planar, then we can conclude the proof. A source code for such a computer program is given in Appendix~\ref{sec:appendix}.
\end{proof}

\begin{figure}
	\centering
	\includegraphics[width= 10cm]{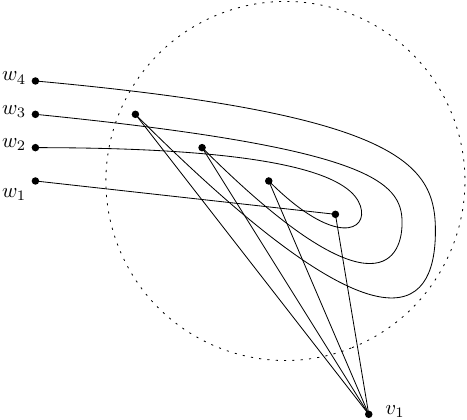}
	\caption{The 1-subdivision of a star drawn as a local thrackle.}
	\label{fig:star}
	\centering
    \bigskip
	\includegraphics[width= 10cm]{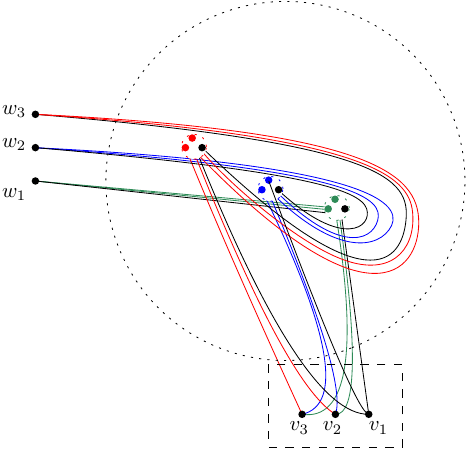}
	\caption{Subdivided $K_{3,3}$ drawn as a local thrackle. The intersections among the edges of the subdivided stars centered at $w_i$ are hidden inside the dotted circles, each of them hiding a drawing that looks like the drawing on Figure \ref{fig:star} (with $w_i$ and $V$ on this figure playing the role of $v_1$ and $W$ on Figure \ref{fig:star}, respectively).}
	\label{fig:k33subdiv}
\end{figure}

\begin{theorem}[Theorem~5.1 in \cite{lovasz1997conway}]\label{thm:notheta3thrackle} 
    A thrackleable graph cannot contain $\Theta_3$ as a subgraph.
\end{theorem}

We show the following strengthening of the above theorem:

\begin{claim}\label{claim:notheta3local} 
    A local thrackleable graph cannot contain $\Theta_3$ as a subgraph.
\end{claim}

\begin{proof}
    Assume on the contrary, then $\Theta_3$ can be drawn as a local thrackle. By Theorem \ref{thm:notheta3thrackle} this is not a thrackle, therefore there are two independent edges that do not cross. The only way this can happen is that there is a subgraph $C_6$ drawn such that a pair of opposite edges do not cross. This is a local thrackle drawing of a $C_6$ which is not a thrackle drawing, contradicting Lemma \ref{lem:c6}, finishing the proof. 
\end{proof}

\begin{claim}\label{claim:knnlocal}
    The $1$-subdivision of the complete graph $K_{n,n}$ is local thrackleable.
\end{claim}

\begin{proof}
    First we show that a $1$-subdivision of a star $S_n$ is local thrackleable. For that see Figure \ref{fig:star}. Note that outside of an area that can be made arbitrarily small (the inside of the dotted circle on the figure), the edges incident to the center can be drawn close to each other and the edges incident to the leaves can be drawn parallel. That is, every $2$-path starting at $v_1$ goes very close to a given $2$-path (say the one joining $v_1$ to $w_1$) and only in a small vicinity of the middle vertex of this given $2$-path do these $2$-paths intersect each other. Note that the cyclic order of the paths around $v_i$ is reversed by the time they reach the $w_i$'s.

    We have a drawing of the 1-subdivision of the star $S_n$ with center $v_1$ and leaves $W=\{w_1,\dots,  w_n\}$. Now we draw $n$ copies of this graph so that the leaves are identified to get the drawing of the graph $G$ which is the $1$-subdivision of $K_{n,n}$. The way we do it is that we put $n-1$ more vertices, $v_2,\dots, v_n$, close to $v_1$ and draw the new edges in the following way. For each $w_i$, we had one $2$-path going to $v_1$ and we need $n-1$ more $2$-paths going to $v_2,
    \dots,v_n$. We draw them the same way as we created the original star, just now $w_i$ is the center and the $2$-path from $w_i$ to $v_1$ is the $2$-path close to which all the other $2$-paths go (and they only intersect in a small vicinity of the middle vertex of the $2$-path from $w_i$ to $v_1$). We denote $V=\{v_1,\dots v_n\}$.

    We draw the edges in the vicinity of $v_1$ such that they intersect pairwise at most once and two edges with a common endpoint $v_i$ do not cross. See Figure \ref{fig:k33subdiv}, the dashed rectangle is the vicinity of $v_1$ where such crossings happen.

    Finally, we are left to prove that this is a local thrackle. It is easy to check we have a simple drawing. It remains to check that no 3-edge path is planar. For this, notice that any $3$-path contains exactly one vertex from $V$ and $W$ respectively. This implies that the $3$ edges of any $3$-path belong to a 1-subdivision of a star that was drawn as a local thrackle. For $v_1$ this follows from our drawing of $S_n$, for other $v_i$ it follows from the fact that these edges go close to the respective edge from $v_1$, and for $w_i$ it again follows from the fact that at $w_i$ we drew the edges in distance two as a local thrackle.
\end{proof}

\begin{proof}[Proof of Proposition~\ref{lt:general}]
    The upper bound directly follows from Theorem \ref{thm:notheta3comb} and Claim \ref{claim:notheta3local}.    The lower bound follows from Claim \ref{claim:knnlocal} by observing that in a $1$-subdivision of $K_{n,n}$ we have $n^2+2n$ vertices and $2n^2$ edges.
\end{proof}

Note that Claim \ref{claim:knnlocal} implies that the $3$-subdivision of $K_n$ is local thrackleable, which implies that the $3$-subdivision of any graph is local thrackleable. Therefore arguments that rely on forbidden subdivided copies of some graph (e.g., planarity arguments) won't work when trying to improve the upper bound in Proposition~\ref{lt:general}.

\subsection{Further examples of local thrackles}
To finish this section, we give some examples of graphs that are not subgraphs of a $1$-subdivision of some $K_{n,n}$ that are nevertheless local thrackleable. As the only possible minimal counterexamples for the thrackle conjecture, the following graphs are especially interesting: two cycles connected by a path (called \emph{dumbbells}), three paths each connecting the same two points (called \emph{theta}), and two cycles sharing a vertex (called \emph{figure-eight}) \cite{thrackleorg}.

First we show that dumbbells in which the cycles are triangles, are local thrackleable.

\begin{claim}
    If $G$ is a graph that contains two vertex-disjoint triangles connected by a path with at least one edges, then $G$ is local thrackleable.    
\end{claim}
\begin{proof}
    For dumbbells with a path of length $1$ or $2$ see Figure \ref{fig:2triangles}. Moreover, there is a well-known operation that replaces an edge by a path with $3$ edges \cite{thrackleorg} (see right side of Figure \ref{fig:2triangles}) that maintains thrackleability.  Clearly, this operation also maintains local thrackleability. Starting from the two base cases and repeatedly using this operation, we can get a local thrackle drawing of $G$ as in the claim for any given path length.
\end{proof}

In the proof of the previous claim, we have seen one operation that produces new local thrackleable graphs from existing ones. We show one more such operation. 
 
\begin{figure}[ht]
	\centering
	\includegraphics[scale=0.8, width= 10cm]{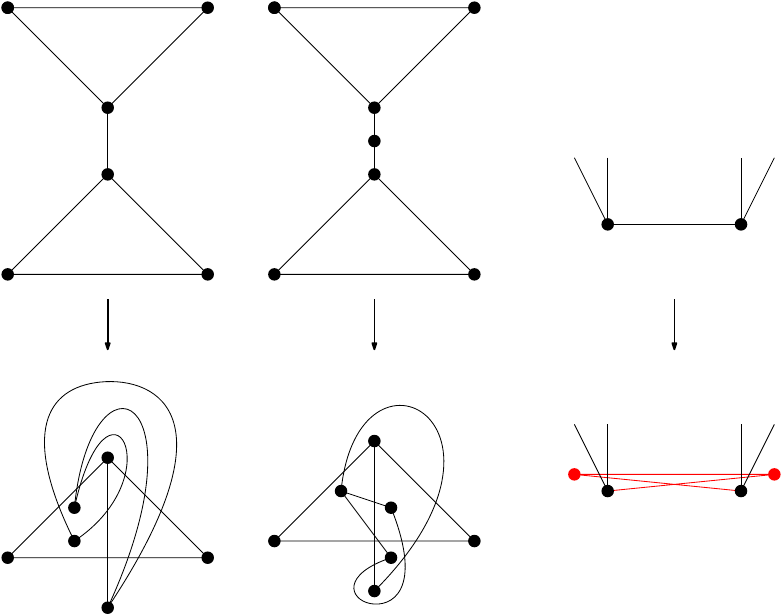}
	\caption{The graphs with two triangles connected by a path are local thrackleable.}
	\label{fig:2triangles}	
\end{figure}

\begin{claim}
    Let $G$ be a local thrackleable graph which has an induced $k$-path for some $k\ge 4$. Then if we add another copy of this path to $G$ with the same endvertices and new inner vertices to get the graph $G'$, then $G'$ is also local thrackleable.
\end{claim}

\begin{proof}
    See Figure \ref{fig:duplicatepath} to see how the new path can be added to a local thrackle drawing of $G$. New points are very close to corresponding old points and new edges go very close to corresponding old edges. Note that in the drawing we wanted to visualize how the new points and edges relate to the old ones, and so for better visibility we have drawn the path of $G$ as a plane path, even though in reality of course it is not, as this was a local thrackle drawing of $G$.
\end{proof}

\begin{figure}[ht]
	\centering
	\includegraphics[scale=0.8, width= 13cm]{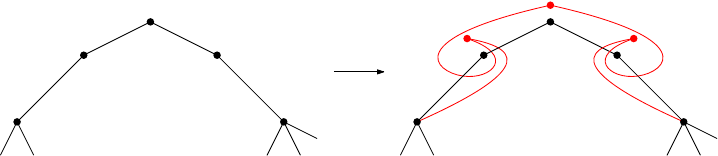}
	\caption{Duplicating a path of length at least $4$.}
	\label{fig:duplicatepath}	
\end{figure}

Using the above claim and that all cycles are local thrackleable (except for $C_4$) among others we get that:
\begin{corollary}
    If $k\ge 4$ and $l\ge 1$, then any theta graph with paths of sizes $k,k,l$ is local thrackleable.
\end{corollary}

\section{No self-intersecting paths of length at most 4}\label{sec:4}

In this section, we prove Theorem \ref{thm:selfcross} by combining the arguments in \cite{PPTT} and a lemma due to Fox, Pach, and Suk \cite{fox2013number} on decomposable families of curves.   We start by proving the following theorem.

\begin{theorem}\label{thmgamma}
    Let $G$ be an $n$-vertex bipartite simple topological graph such that every edge in $G$ crosses a curve $\gamma$ exactly once.  If $G$ does not contain a self-intersecting path of length 4, then $|E(G)| \leq O(n\log n)$.
\end{theorem}

\begin{proof}

Let $V(G) = U\cup V$ be the bipartition of the vertex set such that every edge in $G$ has one endpoint in $U$ and the other in $V$.  Let $x$ and $y$ be the endpoints of $\gamma$.  For each edge $e  = uv$ in $G$, where $u \in U$ and $v \in V$, we say that $e$ is of $\textit{type 1}$ if when starting at vertex $u$ and tracing along $e$ until we reach $\gamma$, and then turning right and tracing along $\gamma$, we reach vertex $x$.  Otherwise (if we reach vertex $y$ instead), we say that $e$ is of $\textit{type 2}$.  Without loss of generality, we can assume that at least half of the edges between $U$ and $V$ are of type 1, since a symmetric argument would follow otherwise.  Let $E'\subset E(G)$ be the edges in $G$ of type 1, where $|E'| = m \geq |E|/2$.  Let $e_1,e_2,\ldots, e_m$ be the edges in $E'$ listed from $x$ to $y$ along $\gamma$, in the order they intersect $\gamma$ (ties are broken arbitrarily).  For each $i \in [m]$, let $e_i = u_iv_i$ with $u_i\in U$, $v_i\in V$ and set \begin{equation*}
    M = \left(\begin{array}{cccc}
    u_1 & u_2 & \cdots & u_m \\
    v_1  & v_2 &  \cdots & v_m
\end{array}\right).
\end{equation*} Hence, $M$ is a $2\times m$ matrix with whose entries are from $U\cup V$, and $u_i \neq v_j$ for all $i,j$.  We now make the following observation.

\begin{claim}
Matrix $M$ does not contain a submatrix of the form \begin{equation*}
    F_1 = \left(\begin{array}{cccc}
    a & b & a & b \\
    \ast  &s & s & \ast
\end{array}\right) \hspace{.5cm}\textnormal{or}\hspace{.5cm}F_2 = \left(\begin{array}{cccc}
    \ast & a & a & \ast \\
   s  & t & s & t
\end{array}\right),  
\end{equation*} where $a\neq b$, $s\neq t$ and $\ast$ stands for arbitrary (not necessarily equal) vertices. 
\end{claim}

\begin{proof}
The first three columns of $F_1$ correspond to a path $P$ of length 3 in $G$ and $P$ must be not self-intersecting, otherwise the whole matrix $F_1$ would correspond to a self-intersecting path of length 4. We can see that $P$ crosses $\gamma$ as in Figure \ref{fig:length4}. Hence, vertex $b$ lies inside the enclosed region bounded by $\gamma$ and the edges corresponding to the first and third column of $F_1$. The fourth column of $F_1$ corresponds to an edge emanating out of vertex $b$ that must cross $\gamma$ (for the first time) outside of this enclosed region. However, since $G$ does not contain a self-intersecting path of length 4, we have a contradiction. A symmetric argument follows for $F_2$. \end{proof}

\begin{figure}[ht]
    \centering
     \includegraphics[width=0.3\linewidth]{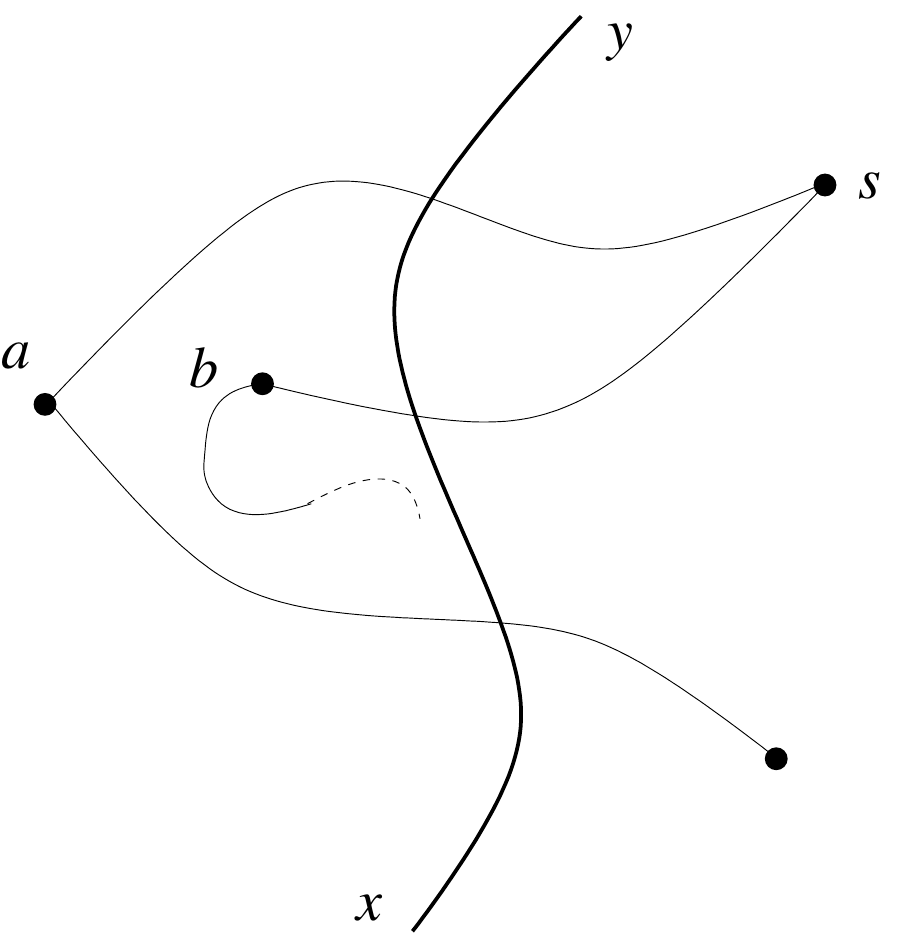}
    \caption{Path of length 4 corresponding to $F_1$, with $\gamma = xy$.}
    \label{fig:length4}
\end{figure}

In \cite{PPTT}, Pach, Pinchasi, Tardos, and T\'oth proved the following.

\begin{lemma}[Pach-Pinchasi-Tardos-T\'oth]
Let $M$ be a $2 \times m$ matrix with at most $n$ distinct entries, all of whose
columns are different. If $M$ has no $2 \times 4$ submatrix of the form $F_1$ or of $F_2$, then $m \leq 17n \log_2 n$.
\end{lemma}

\noindent    Hence, we have $m \leq O(n\log n)$ which implies that $|E(G)| = O(n\log n).$
\end{proof}

Call a collection $\mathcal{C}$ of curves in the plane \textit{decomposable} if there is a partition $\mathcal{C} = \mathcal{C}_1 \cup \cdots\cup \mathcal{C}_t$ such that each $\mathcal{C}_i$ contains a curve which intersects all other curves in $\mathcal{C}_i$, and for $i \neq j$, the curves
in $\mathcal{C}_i$ are disjoint from the curves in $\mathcal{C}_j$ .  We will use the following lemma due to Fox, Pach, and Suk \cite{fox2013number} Lemma~3.2.

\begin{lemma}[Fox-Pach-Suk]\label{lemdecomp}
    There is an absolute constant $c > 0$ such that every collection $\mathcal{C}$ of $m \geq 2$ curves
such that each pair of them intersect in at most $t$ points has a decomposable subcollection of size at least $\frac{cm}{t\log m}$.
\end{lemma}

Finally, we are ready to prove Theorem \ref{thm:selfcross}.  Since every graph has a bipartite subgraph with at least half of its edges, Theorem \ref{thm:selfcross} follows immediately from the following.  

\begin{theorem}
 Every bipartite $n$-vertex simple topological graph $G$ with no self-intersecting path of length 4 has at most $O(n\log^2 n)$ edges.  
\end{theorem}

\begin{proof}
Let $G = (V, E)$ be a bipartite simple topological graph on $n$
vertices with no self-intersecting path of length 4. By Lemma \ref{lemdecomp}, there is a subset $E'\subset E(G)$ such that $|E'| \geq c|E(G)|/ \log |E(G)|$, where $c$ is an
absolute constant and $E'$ is decomposable. Hence, there is a partition\begin{equation*}
    E' = E_1\cup E_2\cup \cdots \cup E_t
\end{equation*}
\noindent such that $E_i$ has an edge $e_i$ that crosses every other edge in $E_i$, and for $i \neq j$, the edges in $E_i$ are disjoint from every edge in $E_j$.  Let $V_i$ be the endpoints of the edges in $E_i$.  By Theorem \ref{thmgamma}, $|E_i| \leq O(|V_i|\log |V_i|)$.  Hence \begin{equation*}
    \frac{c|E(G)|}{\log |E(G)|} \leq \sum\limits_{i = 1}^t |E_i | \leq \sum\limits_{i = 1}^t O(|V_i|\log |V_i|) = O(n\log n).
\end{equation*} Therefore, we have $|E(G)| = O(n\log^2n).$
\end{proof}

\section{Related problems and concluding remarks}\label{sec:remarks}

\noindent \textbf{1.} Two topological graphs $G$ and $H$ are \textit{isomorphic} if there is a homeomorphism of the sphere that transforms $G$ to $H$, where this sphere is the one-point compactification of the plane we draw the graphs on. The proof of Theorem~\ref{main} finds more than a weak-isomorphic copy of $C_{k,k}$ any simple drawing of a large enough complete graph. The drawing of $K_{k,k}$ we find as a subgraph agrees with $C_{k,k}$ in its \textit{extended rotation system}, that is, the cyclic orders of edges emanating from all vertices and crossings. According to Theorem~1 in the recent paper~\cite{aichholzer2023drawings}, this determines $H$ up to triangle flips followed by an isomorphism. Here, a \textit{triangle flip} refers to the operation of moving one edge of a triangular cell formed by three pairwise crossing edges over the opposite crossing of the two other edges.

\medskip

\noindent \textbf{2.}  The following 4-uniform hypergraph Ramsey problem was also studied independently by Negami \cite{Ne} and Mathias Schacht (private communication).  Let $S_{a,b} = (U, V, E)$ be the 4-uniform hypergraph with the vertex set the disjoint union of $U$ and $V$, where $|U| = a$ and $|V| = b$, such that $E(S_{a,b}) = \{\{x, y,z,w\} : x\ne y, z\ne w, x,y \in U, \textnormal{ and } z,w \in V \}.$  What is the minimum integer $r(S_{n,n})  = N$ such that for any red/blue coloring of the edges of $S_{N,N}$, there is a monochromatic copy of $S_{n,n}$?  Following the arguments given in Section 2, one can show that $r(S_{n,n}) < 2^{2^{O(n^2)}}$, while a standard probabilistic argument shows that $r(S_{n,n}) > 2^{\Omega(n^3)}$.

\bigskip\noindent{\bf Acknowledgement}
This material is based upon work supported by the National Science Foundation under Grant No. DMS-1928930, while the authors were in residence at the Simons Laufer Mathematical Sciences Institute in Berkeley, California, during the 2025 Extremal Combinatorics Program.

\bibliographystyle{abbrv}
{\footnotesize
\bibliography{main}}

\appendix

\section{Program for Lemma~\ref{lem:c6}}\label{sec:appendix}

The Python program was ran in June 2025 on Google Colab platform. The execution time was about 9 seconds. In the program, the vertices of $C_6$ are denoted as $0,1,2,3,4,5$, and the edges of $C_6$ are denoted as $(0,1), (1,2), (2,3), (3,4), (4,5), (5,0)$. The edge $(i,i+1)$ is directed from $i$ to $i+1$, where the indices are considered modulo $6$. Without loss of generality, we assume the edges $(0,1)$ and $(3,4)$ do not cross to avoid producing a thrackle.

\begin{lstlisting}[breaklines=true, frame=tb, language=Python, basicstyle={\footnotesize}]
# NetworkX is a Python package that includes a planarity testing algorithm. Itertools is a standard Python package that helps us to generate permutations.

import networkx
import itertools

# We enumerate all cases on whether (1,2) crosses (4,5) and whether (2,3) crosses (5,0).

for cross1245 in [True,False]:
  for cross2350 in [True,False]:

    # EC records the set of edges crossed by a given edge. For example, EC[3]==[(5,0),(1,2)] means the edge (3,4) crosses (5,0) and (1,2), which is required by the local thrackle condition.

    EC = [[(2,3),(4,5)],
          [(3,4),(5,0)],
          [(4,5),(0,1)],
          [(5,0),(1,2)],
          [(0,1),(2,3)],
          [(1,2),(3,4)]]

    # We modify EC based on whether (1,2) crosses (4,5) and whether (2,3) crosses (5,0).

    if cross1245:
      EC[1].append((4,5))
      EC[4].append((1,2))
    if cross2350:
      EC[2].append((5,0))
      EC[5].append((2,3))

    # We enumerate all permutations for every entry of EC.

    for l0 in itertools.permutations(EC[0]):
      for l1 in itertools.permutations(EC[1]):
        for l2 in itertools.permutations(EC[2]):
          for l3 in itertools.permutations(EC[3]):
            for l4 in itertools.permutations(EC[4]):
              for l5 in itertools.permutations(EC[5]):

                # The ordered crossing list is denoted as l. For example, l[0]==[(4,5),(2,3)] means along the edge (0,1), from 0 to 1, the first crossing is with (4,5), and the second crossing is with (2,3).

                l = [l0,l1,l2,l3,l4,l5]

                # We create H and its vertices. For example, the vertex created from the crossing point between (3,4) and (5,0) is denoted as a tuple ((3,4),(5,0)) sorted lexicographically. The first subdivision point along the edge (3,4) is denoted as ((3,4),0) by zero-based numbering.

                H = networkx.Graph()
                H.add_nodes_from([0,1,2,3,4,5])
                for i in range(6):
                  for e in EC[i]:
                    H.add_node(tuple(sorted([(i,(i+1)%6),e])))
                for i in range(6):
                  for j in range(len(l[i])-1):
                    H.add_node(((i,(i+1)%6),j))

                # We create the edges for H at the planarization step and the subdivision step.

                for i in range(6):
                  H.add_edge(i, tuple(sorted([(i,(i+1)%6),l[i][0]])))
                  for j in range(len(l[i])):
                    if  j == len(l[i]) - 1:
                      H.add_edge(tuple(sorted([(i,(i+1)%6),l[i][j]])) ,(i+1)%6)
                    else:
                      H.add_edge(tuple(sorted([(i,(i+1)%6),l[i][j]])) , ((i,(i+1)%6),j))
                      H.add_edge(((i,(i+1)%6),j) , tuple(sorted([(i,(i+1)%6),l[i][j+1]])))
                
                # We create the edges for H at the bracing step.

                for i in range(6):
                  for j in range(len(l[i])):
                    v = tuple(sorted([(i,(i+1)%6),l[i][j]]))
                    if j == 0:
                      u = i
                      w = ((i,(i+1)%6),j)
                    elif j == len(l[i]) - 1:
                      u = ((i,(i+1)%6),j-1)
                      w = (i+1)%6
                    else:
                      u = ((i,(i+1)%6),j-1)
                      w = ((i,(i+1)%6),j)
                    for x in H.neighbors(v):
                      if x != u and x != w:
                        H.add_edge(u,x)
                        H.add_edge(w,x)

                # We report if H is planar.

                if networkx.is_planar(H):
                  print("This message means that we have found a counterexample, otherwise the program will be as silent as Wittgenstein!")
\end{lstlisting}

\end{document}